\documentclass[10pt]{article}
\usepackage{amsfonts}
\usepackage[pdftex]{graphicx}
\usepackage{amsmath}
\usepackage{amssymb}
\usepackage{amsthm}
\usepackage{caption}
\usepackage{subcaption}
\usepackage[top=1in, bottom=1in, left=1in, right=1in]{geometry}
\usepackage[driverfallback=pdfmx,,colorlinks,bookmarksopen,bookmarksnumbered,citecolor=blue, urlcolor=blue]{hyperref}

\newenvironment{keywords}{
\list{}{\advance\topsep by0.35cm\relax\small
\leftmargin=1cm
\itemindent\listparindent
\rightmargin\leftmargin}\item[\hskip\labelsep
\bfseries Keywords:]}
{\endlist}
\newtheorem{theorem}{Theorem}[section]

\newtheorem{lemma}[theorem]{Lemma}

\newtheorem{remark}[theorem]{Remark}

\begin{document}

\title{Global existence for a strongly coupled reaction diffusion system}

\author{ \textsc{\ Said Kouachi,\thanks{%
E-mail: kouachi@hotmail.com 
Department of Mathematics, College of Science, 
Qassim University, P.O.Box 6644, Al-Gassim, Buraydah 51452, Kingdom of Saudi Arabia.
}
\hspace{0.1cm} Kamuela E. Yong,\thanks{%
E-mail: kamuela.yong@asu.edu
Simon A. Levin Mathematical, Computational and Modeling Sciences Center $\&$ School of Mathematical and Statistical Sciences, Arizona State University, Tempe, AZ, 85287-1904,USA.} \hspace{0.1cm} 
\& Rana D. Parshad\thanks{
E-mail: rparshad@clarkson.edu
Department of Mathematics, Clarkson University, Potsdam, NY 13699, USA.
}\vspace{0.3cm}} \\
}

\date{}
\maketitle

\begin{abstract}
In this work we use functional methods to prove the boundedness and global existence of solutions for a class of strongly coupled parabolic systems. We apply the results to deduce the global existence of solutions for a classic Shigesada-Kawasaki-Teramoto (SKT) type model for an extended range of the self-diffusion and cross-diffusion coefficients than those available in the current literature. We perform numerical simulations in 2D, via a spectral Galerkin method to verify our global existence results, as well as to visualize the dynamics of the system.
\end{abstract}


\section{Introduction}

Modeling species interactions occupies a central theme in spatial ecology and mathematical biology. Many of these models take the form of reaction diffusion equations, where the reaction terms describe both inter-species and intra-species dynamics, including factors such as birth and death. The diffusion represents random spatial movement of the species. Recently there has been a lot of activity concerning models where the diffusion of one species may be influenced by another species. These models come under the class of strongly coupled parabolic systems \cite{J06,L96} and are also known in the literature as cross-diffusion systems. Although these models may be more realistic from a biological or modeling point of view (as they naturally incorporate inter-species and intra-species competition), they are more difficult to handle mathematically than their constant coefficient, pure diffusion counterparts. Their classical applications are immense and there are new potential applications ranging from bee pollination \cite{Yong} to pedestrian flow \cite{Berres}. Among the first efforts to model cross-diffusion was the model system proposed in \cite{Shigesada} by Shigesada, Kawasaki and Teramoto. This model and its variants are now known as SKT models. For a detailed steady state analysis of this model one can see \cite{L96}. In the current work, we consider the following cross-diffusion system which is a generalized form of the SKT model, 
\begin{equation}
\label{eq:1.1}
\frac{\partial u}{\partial t}-\nabla \left[ P^{u}\nabla u+P^{v}\nabla v%
\right] =f(u,v)\;=u\left( a_{1}-b_{1}u+c_{1}v\right) \;\;\;\text{in }\mathbb{%
R}^{+}\times \Omega ,  
\end{equation}
\begin{equation}
\label{eq:1.2}
\frac{\partial v}{\partial t}-\nabla \left[ Q^{u}\nabla u+Q^{v}\nabla v%
\right] =g(u,v)=v\left( a_{2}+b_{2}u-c_{2}v\right) \;\;\;\;\text{in }\mathbb{%
R}^{+}\times \Omega ,  
\end{equation}
with the boundary conditions
\begin{equation}
\label{eq:1.3}
\frac{\partial u}{\partial \eta }=\frac{\partial v}{\partial \eta }=0\;\;\;\;%
\text{on }\mathbb{R}^{+}\times \partial \Omega ,  
\end{equation}
and the initial data
\begin{equation}
\label{eq:1.4}
u(0,x)=u_{0}(x),\qquad v(0,x)=v_{0}(x)\;\;\;\;\text{in}\;\Omega , 
\end{equation}
where
\begin{equation}
\label{eq:1.5}
\left\{ 
\begin{array}{c}
P^{u}=d_{1}+\alpha _{11}u+\alpha _{12}v,\ \ \ P^{v}=b_{11}u, \\ 
Q^{u}=b_{22}v,\ \ \ Q^{v}=d_{2}+\alpha _{21}u+\alpha _{22}v.%
\end{array}%
\right.   
\end{equation}
Also, to simplify the notations, we let
\begin{equation*}
P=P^{u}\nabla u+P^{v}\nabla v,\ \ \ \ Q=Q^{u}\nabla u+Q^{v}\nabla v.
\end{equation*}%
The domain $\Omega $ is an open bounded domain of class $\mathbb{C}^{1}$ in $%
\mathbb{R}^{N}$, with boundary $\partial \Omega $ and $\dfrac{\partial }{%
\partial \eta }$ denotes the outward normal derivative on $\partial \Omega $%
. The components $u$ and $v$ are the solution to system \eqref{eq:1.1}-\eqref{eq:1.2}, are
nonnegative functions, and represent the population densities of the prey
and predator species, respectively. $d_{1}$ and $d_{2}$ are positive
constants representing the random diffusion rates of the two species
respectively. The initial data are functions in $W^{1,p}\left( \Omega
\right) \ \left( p>N\right) $ and assumed to be nonnegative which imply,
via the maximum principle \cite{Smoller}, the positivity of
the solution on its interval of existence.
The constants $a_{i}, b_{i}, c_{i} \left( i=1,2\right)$  are any real numbers with $b_{1},c_{2}>0$, where $b_{2}$ and $c_{1}$
 are either nonpostive or both positive, and sufficiently small. Also we may have $c_1 b_2 < 0$, with one positive and one negative, with the positive one sufficiently small in \eqref{eq:3.38}, \eqref{eq:3.381}. The constants $\alpha
_{ij},  b_{ii} \left( i,j=1,2\right) $ are nonnegative constants: $\alpha
_{11}$ and $\alpha _{22}$ are the self-diffusion rates, $\alpha _{12}$ and $%
\alpha _{21}$ are the cross-diffusion rates and $b_{11}, b_{22}$ are the
gradient cross-diffusions. When $\alpha_{12}=b_{11}$, $\alpha_{21}=b_{22}$ and $b_{1} < 0$, $c_{2}< 0$ then \eqref{eq:1.1}-\eqref{eq:1.2} is in fact the SKT model \cite{Shigesada}.

However when the reaction diffusion systems include cross-diffusion as in
 \eqref{eq:1.1}-\eqref{eq:1.2}, there are not many results: When $d_{1}>$ $d_{2},\ \alpha
_{12}=\alpha _{22}=0,\ a_{1}=b_{1}=c_{2}=0,$ an analogous system describing
epidemics, M. Kirane and S. Kouachi established in a series of papers 
\cite{Kirane-Kouachi1,Kirane-Kouachi2,Kirane-Kouachi3} studying global
existence and asymptotics. Shim \cite{Shim1} proved the existence of global
solutions to   \eqref{eq:1.1}-\eqref{eq:1.4} in the space dimension $N=1$, when $d_{1}=d_{2}$
and $\alpha _{11}=\alpha _{22}=0$ or when $0<\alpha _{21}<8\alpha _{11}$ and 
$0<\alpha _{12}<8\alpha _{22}.$ Also in the space dimension $N=1,$ the
same author established global existence when $\alpha _{21}=0.$
Recently,  Xu in \cite{Sheng} showed global existence when $\alpha
_{12}=0$ in two cases: when $\alpha _{11}=0$ or when $\alpha _{11}>0$ and $
N<10.$ The condition $\alpha _{12}=0$ implies by simple application of the maximum principle, the uniform boudedness of $
u(t,.)$ on $[0,T_{\max }[\times \Omega $, where $T_{\max }$ denotes the
eventual blowing-up time in $\mathbb{L}^{\infty }(\Omega ).$
There is a large literature on this class of models, and the interested reader is referred to \cite{Kuto, Loum, Lou98, Matano, Mimura, Pao, Ryu1, Ryu2}. The global existence of weak solutions under natural conditions is settled \cite{CJ04, GJ03}. Recently a number of sophisticated tools that include the use of Compatano and Morrey spaces, as well as Entropy functional techniques have been devised and applied to these problems. For details the reader is referred to \cite{J06,Le-Ngu,LN6}. Also, there have been recent efforts to prove global boundedness of weak solutions to cross diffusion systems, using an entropy functional approach \cite{JZ14}.
On the more applied side, Yong et al.~\cite{Yong} considered a diffusion system, modeling the interactions between honey bees and solitary bees using a form of the SKT model. They derive a number of interesting conclusions from their work by restricting the self-diffusion and cross-diffusion coefficients and compare their simulation results to field data \cite{Yong}

We focus on \cite{D. Le}, where the author considered \eqref{eq:1.1}-\eqref{eq:1.4}, however he imposed the restriction that the reaction terms should be negative if either $u$ or $v
$ is sufficiently large.  This is not necessary in our case, since the second
reaction term can be large for $u\gg v$. In \cite{D. Le} the author proves global existence when $b_{2}
$ and $c_{1}$ are non-positive, under the following conditions
\begin{equation}
\label{eq:1.6}
\left. 
\begin{array}{l}
(i)\ \alpha _{11}\alpha _{22}+\alpha _{12}\alpha _{21}-b_{11}b_{22}\geq 0,
\\ 
(ii)\ \alpha _{22}-\alpha _{12}>b_{11}, \\ 
(iii)\ \alpha _{11}-\alpha _{21}>b_{22}.%
\end{array}%
\right.  
\end{equation}
When $\alpha _{ij}=0$ ($i,j=1,2$), and $b_1$ and $c_2$ are of opposite signs, the system is the well-known Lotka-Volterra
prey-predator model. In this case and for more general reactions and
boundary conditions many results have been obtained. For more details see \cite{Alikakos, Hollis-Martin-Pierre, Masuda, Morgan}.

Our primary contributions in this paper are the following:
\\
\vspace{1mm}
\\
1) We prove boundedness of solutions to \eqref{eq:1.1}-\eqref{eq:1.4}
via Theorem \ref{thm:t21}, when $\alpha _{11}>\alpha _{21}\ $and $\alpha _{22}>\alpha _{12}$ under the
following condition
\begin{equation}
\label{eq:1.7}
\left( \alpha _{11}-\alpha _{21}\right) \left( \alpha _{22}-\alpha
_{12}\right) >b_{11}\text{ }b_{22},  
\end{equation}%
It is clear that our Condition \eqref{eq:1.7} is weaker than \eqref{eq:1.6}. We next use the everywhere regularity result of bounded solutions
to  \eqref{eq:1.1}-\eqref{eq:1.4}, via Theorem \ref{thm:t11} from \cite{D. Le} to: 
\\
\vspace{1mm}
\\
2) Deduce global existence of solutions, without restriction on the
space dimension, to \eqref{eq:1.1}-\eqref{eq:1.4}, via our new Condition \eqref{eq:2.1}. This is more general than \cite{D. Le}. This is accomplished  via Theorem \ref{thm:t22}.
\\
\vspace{1mm}
\\
We then perform numerical simulations under the following strategy: 
\\
\vspace{1mm}
\\
3) We choose parameters that satisfy Condition \eqref{eq:2.1}, whilst violating \eqref{eq:1.6}. With this choice of parameters we simulate \eqref{eq:1.1}-\eqref{eq:1.4} for a host of different initial conditions. Here we demonstarte that there are a host of initial conditions, for which we have globally existing solutions, for which the results of \cite{D. Le} are not applicable, but our result Theorem \ref{thm:t22} is indeed applicable.
\vspace{1mm}

For the benefit of the reader we recap the H\"{o}lder continuity result via Theorem \ref{thm:t11}, from \cite{D. Le}

\begin{theorem}[\cite{D. Le}]
\label{thm:t11}
Assume that $d_{i},\ \alpha _{ij},\ b_{ii}>0,\ i,\ j=1,\ 2$ and%
\begin{equation}
\label{eq:1.8}
\alpha _{11}\alpha _{22}+\alpha _{12}\alpha _{21}-b_{11}b_{22}\geq 0. 
\end{equation}%
Let $V_{1}=\left( \alpha _{11}-\alpha _{21}-b_{22}\right) $ and $%
V_{2}=\left( \alpha _{22}-\alpha _{12}-b_{11}\right) $. We assume that%
\begin{equation}
\label{eq:1.9}
\left. 
\begin{array}{l}
(i)\ V_{1}=0\text{ and }V_{2}\neq 0\text{, and vice versa, or} \\ 
(ii)\ V_{1}V_{2}>0\text{, or} \\ 
(iii)\ \left( d_{1}-d_{2}\right) \left( V_{2}-V_{1}\right) \left[ \left(
\alpha _{11}-\alpha _{21}\right) \left( \alpha _{22}-\alpha _{12}\right)
-b_{11}\text{ }b_{22}\right] >0.%
\end{array}%
\right.  
\end{equation}%
Then bounded positive weak solutions to   \eqref{eq:1.1}-\eqref{eq:1.4}  are 
H\"{o}lder continuous.
\end{theorem}

\begin{remark}
This everywhere regularity result of bounded solutions to
 \eqref{eq:1.1}-\eqref{eq:1.4} obtained in \cite{D. Le} is very interesting. However, the author was
hindered, by not being able to show the boundedness and global existence when $V_1 = 0$ or $V_2 = 0$. He was only able to prove global existence in the case when $V_{1}>0$ and $V_{2}>0.$ Furthermore we remark that
the second and third inequalities in Conditions \eqref{eq:1.6} (which was assumed to prove global existence in \cite{D. Le}) imply the
first inequality when it is a strict inequality, and contradict it in the case of
equality.
\end{remark}

In Section 2 we state our main results, then we present our
proofs in Section 3. Lastly, in Section 4 we present numerical simulations, that support our results Theorems \ref{thm:t21}, \ref{thm:t22}. 

\section{\textbf{Statement of the main results}}

\bigskip For initial conditions in $W^{1,p}\ \left( p>N\right) $, it was
proved in \cite{Amman} that solutions to problems more general than
 \eqref{eq:1.1}-\eqref{eq:1.4} exist locally in time. For the global existence, it was proved
that solutions to \eqref{eq:1.1}-\eqref{eq:1.4} exist globally in time if one has control on both of their $\mathbb{L}^{\infty }$ and H\"{o}lder norms. This is in
contrast to classical reaction diffusion systems (see \cite
{Friedman, Henry, Pazy, Rothe, Smoller}) where one need only control the L$^{\infty }$
norms of solutions. Counterexamples in \cite{John-Stara} confirmed this:
bounded weak solutions of certain strongly coupled systems may blow up in
finite time.

Our aim is to construct a polynomial Lyapunov functional (see \cite{Kouachi3,Kouachi1,Kouachi2,Kouachi4,D. Le,Par-Kou-Gut}) 
depending on the solution $(u,v)$ of system  \eqref{eq:1.1}-\eqref{eq:1.2}. We then use functional methods to
derive their L$^{\infty }$ bounds, via this functional. Next, we apply known results to show their H\"{o}%
lder continuity, and thus deduce their global existence. More precisely, we will
use a quadratic form according to the solution $(u,v)$ of system  \eqref{eq:1.1}-\eqref{eq:1.4} 
to prove the following result concerning the boundedness of solutions

\begin{theorem}
\label{thm:t21}
Consider $\alpha _{11}>\alpha _{21}\ $and $\alpha _{22}>\alpha _{12}$, then under the
following condition
\begin{equation*}
\left( \alpha _{11}-\alpha _{21}\right) \left( \alpha _{22}-\alpha
_{12}\right) >b_{11}\text{ }b_{22},  
\end{equation*}
positive solutions of problem  \eqref{eq:1.1}-\eqref{eq:1.4} 
are bounded on $\left[ 0,T_{\max }\right[ $ by constants depending on
initial data and the reaction terms.
\end{theorem}

Then, we will use Theorem \ref{thm:t11} to deduce that bounded solutions of system
 \eqref{eq:1.1}-\eqref{eq:1.2}  are H\"{o}lder continuous. We will give under stronger
assumptions than those in Theorem \ref{thm:t11}, the proof of the following

\begin{theorem}
\label{thm:t22}
Assume that
\begin{equation}
\label{eq:2.1}
\left. 
\begin{array}{l}
(i)\ V_{1}=0\text{ and }V_{2}>0\text{, and vice versa, or} \\ 
(ii)\ \alpha _{11}-\alpha _{21}>b_{22}\text{ and }\alpha _{22}-\alpha
_{12}>b_{11}\text{, or} \\ 
(iii)\ d_{1}>d_{2},\ \alpha _{11}-\alpha _{21}>b_{22},\text{ }\alpha
_{22}-\alpha _{12}<b_{11}\text{ and \eqref{eq:1.7}, or} \\ 
(iv)\ d_{1}<d_{2},\ \alpha _{11}-\alpha _{21}<b_{22},\text{ }\alpha
_{22}-\alpha _{12}>b_{11}\text{ and \eqref{eq:1.7},}%
\end{array}
\right.  
\end{equation}
then weak solutions with nonnegative initial data to  \eqref{eq:1.1}-\eqref{eq:1.4}  are
classical and exist globally.
\end{theorem}

\begin{remark}
Note that our Conditions  \eqref{eq:2.1}  are weaker than  \eqref{eq:1.6}. For
example, when $d_{1}>d_{2}$ we can take $0<\alpha _{11}-\alpha _{21}<b_{22}$
and $\left( \alpha _{22}-\alpha _{12}\right) >b_{11}$ such that Condition (iii) in \eqref{eq:1.6} is
not satisfied, and the results of \cite{D. Le} are not applicable. However, our Condition (iv) in \eqref{eq:2.1} is still satisfied, and we have global existence.
\end{remark}

We consider here systems of two equations with homogeneous
Neumann boundary conditions, but our main results are applicable to
those of more equations and with suitable other boundary conditions.

For a given function $w,$ we will denote by $w_{+}$ the nonnegative part of $w$,
$sup\{w,0\}$ and with a slight abuse of notation, we will write $H_{u}=%
\dfrac{\partial H}{\partial u}(u,v),\ H_{uv}=\dfrac{\partial ^{2}H}{\partial
u\partial v}(u,v),\ \nabla H=\nabla _{x}\left( H\left( t,x\right) \right) $
and so on.

\section{\textbf{Proofs}}

In this section, we study the boundedness of solutions to  \eqref{eq:1.1}-\eqref{eq:1.4},
for this purpose we consider the following quadratic form%
\begin{equation}
\label{eq:3.1}
H\left( u,v\right) =\tfrac{1}{2}\lambda u^{2}+uv+\tfrac{1}{2}\mu v^{2}, 
\end{equation}
where $\lambda $ and $\mu $ are positive constants such that $\lambda \mu
=K^{2}$ for some constant $K$ which we suppose $>1$ to assure the positive
definiteness of $H\left( u,v\right)$. We define a subset of $\mathbb{R}%
_{+}^{2}$ as a neighborhood of a local solution as follows%
\begin{equation}
\label{eq:3.2}
\Gamma =\left\{ (u\left( t,x\right) ,v\left( t,x\right) ):\ \ 0<t<T_{\max
},\ \ \ x\in \Omega \right\} .  
\end{equation}%
For the proof of the Theorem 2.1 on the boundedness of solutions to
 \eqref{eq:1.1}-\eqref{eq:1.4} , we need some Lemmas and apply the following (see D. Le \cite%
{D. Le})

\begin{theorem}
\label{thm:t31}
 If there exist positive real numbers $\lambda _{1}$and $C_{0}$ such that%
\begin{equation}
\label{eq:3.3}
\left\{ 
\begin{array}{c}
P \cdot\nabla H_{u}+Q\cdot\nabla H_{v}\geq 0, \\ 
\left( H_{u} \cdot P+H_{v} \cdot Q\right) \nabla H\geq \lambda _{1}\left\vert \nabla
H\right\vert ^{2},%
\end{array}%
\right.   
\end{equation}
and
\begin{equation}
\label{eq:3.4}
H_{u}f+H_{v}g\leq 0,  
\end{equation}
for all \ $(u,v)\in \Gamma \cap \left\{ (u,v):\ H(u,v)>C_{0}\right\} $, then
the solution $(u,v)$ of Problem  \eqref{eq:1.1}-\eqref{eq:1.4}  is bounded.
\end{theorem}

\begin{proof}
We use the following functional%
\begin{equation}
\label{eq:3.5}
L(t)=\frac{1}{2}\int\limits_{\Omega }\left[ (H-C)_{+}\right] ^{2}dx, 
\end{equation}
where $C>\max \left\{ C_{0},\ H_{0}\right\} $ and $H_{0}=\sup_{x\in \Omega
}H(u_{0}\left( x\right) ,v_{0}\left( x\right) ).$

Differentiating $L$ with respect to $t$ yields
\begin{equation}
\label{eq:3.6}
\left. 
\begin{array}{rcl}
L^{\prime }(t)&=&\displaystyle\int\limits_{\Omega }\left[ H_{u}\nabla P+H_{v}\nabla Q\right]
(H-C)_{+}dx+\int\limits_{\Omega }\left\{ H_{u}f(u,v)+H_{v}g(u,v)\right\}
(H-C)_{+}dx \\ 
&:=&I+J,%
\end{array}%
\right.  
\end{equation}
where
\begin{equation}
\label{eq:3.7}
I=\int\limits_{\Omega }(H-C)_{+}\left[ H_{u}\nabla P+H_{v}\nabla Q\right] dx,
\end{equation}
 and
\begin{equation}
\label{eq:3.8}
J=\int\limits_{\Omega \cap \left\{ H>C\right\} }(H-C)_{+}\left\{
H_{u}f(u,v)+H_{v}g(u,v)\right\} dx.  
\end{equation}%
Then by simple application of Green's formula with the boundary conditions
\eqref{eq:1.3}, we get
\begin{equation}
\label{eq:3.9}
\left. 
\begin{array}{rcl}
I&=&\displaystyle-\int\limits_{\Omega }\left\{ \nabla \left[ H_{u}(H-C)_{+}\right]
\cdot P+\nabla \left[ H_{v}(H-C)_{+}\right]  \cdot Q\right\} dx \\ 
&=&\displaystyle-\int\limits_{\Omega \cap \left\{ H>C\right\} }\left\{ \left(
H_{u}P+H_{v}Q\right) \cdot \nabla H+\left( P \cdot \nabla H_{u}+Q \cdot \nabla H_{v}\right)
(H-C)_{+}\right\} dx.%
\end{array}%
\right.  
\end{equation}
Using  \eqref{eq:3.3}  we get
\begin{equation*}
I\leq -\lambda _{1}\int\limits_{\Omega \cap \left\{ H>C\right\} }\left\vert
\nabla H\right\vert ^{2}dx.
\end{equation*}
From  \eqref{eq:3.4}, we have $J\leq 0$, then by integration with respect to $t$ we get
\begin{equation}
\label{eq:3.10}
\left.\int\limits_{\Omega } \left[ (H-C)_{+}\right] ^{2}dx\right\vert
_{0}^{t}+\lambda _{1}\underset{0}{\overset{t}{\int }}\int\limits_{\Omega
\cap \left\{ H>C\right\} }\left\vert \nabla H\right\vert ^{2}dxdt\leq 0. 
\end{equation}
Since $(H-C)_{+}=0$ when $t=0$, we deduce $(H-C)_{+}=0$ for all $t\in \left(
0,T\right) $ and this gives $H\leq C$ on $\left( 0,T\right) \times \Omega .$
But $H\left( u,v\right) \geq \frac{\lambda }{2}\left( u+\frac{v}{\lambda }%
\right) ^{2}$, for $u, v\geq 0,$ we conclude that the solution is bounded
by some constant depending on $C_{0}$ and the initial data.
\end{proof}

\bigskip First, Condition \eqref{eq:3.4}  is assumed by the following

\begin{lemma}
\label{lem:l32}
Under Condition  \eqref{eq:1.7}, the first inequality in  \eqref{eq:3.3} is satisfied
\end{lemma}

\begin{proof}
The first Condition \eqref{eq:3.3} is satisfied if we prove the positivity of the
following quadratic form in $\nabla u$ and $\nabla v:$ 
\begin{equation}
\label{eq:3.11}
\left. 
\begin{array}{rcl}
\Psi &=&P \cdot \nabla H_{u}+Q \cdot \nabla H_{v} \\ 
&=&\left( P^{u}H_{uu}+Q^{u}H_{uv}\right) \left\vert \nabla u\right\vert
^{2}+\left[ H_{uu}P^{v}+\left( P^{u}+Q^{v}\right) H_{uv}+Q^{u}H_{vv}%
\right] \nabla u \cdot \nabla v \\ 
&&+\left( P^{v}H_{uv}+Q^{v}H_{vv}\right) \left\vert \nabla v\right\vert ^{2},%
\end{array}
\right.  
\end{equation}
which can be written as follows
\begin{equation}
\label{eq:3.12}
\Psi =u\Psi _{u}+v\Psi _{v}+\Psi _{d}, 
\end{equation}
where $\Psi _{u},\ \Psi _{v}$ and $\Psi _{d}$ are the following quadratic
forms%
\begin{equation}
\label{eq:3.13}
\left. 
\begin{array}{rcl}
\Psi_{u}&=&\alpha _{11}H_{uu}\left\vert \nabla u\right\vert ^{2}+\left[
b_{11}H_{uu}+\left( \alpha _{11}+\alpha _{21}\right) H_{uv}\right]
\nabla u \cdot \nabla v+\left( b_{11}H_{uv}+\alpha _{21}H_{vv}\right) \left\vert
\nabla v\right\vert ^{2}, \\ 
\Psi_{v}&=&\left( \alpha _{12}H_{uu}+b_{22}H_{uv}\right) \left\vert
\nabla u\right\vert ^{2}+\left[ \left( \alpha _{12}+\alpha _{22}\right)
H_{uv}+b_{22}H_{vv}\right] \nabla u \cdot \nabla v+\alpha _{22}\left\vert
H_{vv}\nabla v\right\vert ^{2}, \\ 
\Psi _{d}&=&d_{1}H_{uu}\left\vert \nabla u\right\vert ^{2}+\left(
d_{1}+d_{2}\right) H_{uv}\nabla u \cdot \nabla v+d_{2}H_{vv}\left\vert \nabla
v\right\vert ^{2}.\ 
\end{array}%
\right.   
\end{equation}
By applying each of the above quadratic forms to the inequality
\begin{equation}
\label{eq:3.14}
A\left\vert \nabla u\right\vert ^{2}+B\nabla u \cdot \nabla v+C\left\vert \nabla
v\right\vert ^{2}\geq -\left( \frac{\Delta }{8C}\left\vert \nabla
u\right\vert ^{2}+\frac{\Delta }{8A}\left\vert \nabla v\right\vert
^{2}\right) ,  
\end{equation}
for positive numbers $A,\ B$ and $C,$ with%
\begin{equation*}
\Delta =B^{2}-4AC,
\end{equation*}
we get%
\begin{equation}
\Psi _{u}\geq -\left( \frac{\left\vert \nabla u\right\vert ^{2}}{8C_{u}}+%
\frac{\left\vert \nabla v\right\vert ^{2}}{8A_{u}}\right) \Delta
_{u}:=-\left( \frac{\left\vert \nabla u\right\vert ^{2}}{8\left( b_{11}+\mu
\alpha _{21}\right) }+\frac{\left\vert \nabla v\right\vert ^{2}}{8\alpha
_{11}\lambda }\right) \Delta _{u},  
\end{equation}%
where%
\begin{equation}
\label{eq:3.16}
\Delta _{u}=\left( b_{11}\lambda -\alpha _{11}+\alpha _{21}\right)
^{2}-4\alpha _{11}\alpha _{21}\left( K^{2}-1\right) .  
\end{equation}%
If we choose $\lambda $ such that%
\begin{equation}
\label{eq:3.17}
\left( b_{11}\lambda -\alpha _{11}+\alpha _{21}\right) ^{2}-4\alpha
_{11}\alpha _{21}\left( K^{2}-1\right) <0,  
\end{equation}%
then the positivity of the quadratic form $\Psi _{u}$, amounts to the
following condition%
\begin{equation}
\label{eq:3.18}
0<\lambda <\frac{\left( \alpha _{11}-\alpha _{21}\right) +2\sqrt{\alpha
_{21}\alpha _{11}\left( K^{2}-1\right) }}{b_{11}}.  
\end{equation}%
For the second quadratic form, we have
\begin{equation}
\label{eq:3.19}
\Psi _{v}\geq -\left( \frac{\left\vert \nabla u\right\vert ^{2}}{8C_{v}}+%
\frac{\left\vert \nabla v\right\vert ^{2}}{8A_{v}}\right) \Delta
_{v}:=-\left( \frac{\left\vert \nabla u\right\vert ^{2}}{8\alpha _{22}\mu }+%
\frac{\left\vert \nabla v\right\vert ^{2}}{8\left( b_{11}\lambda +\alpha
_{21}\right) }\right) \Delta _{v},  
\end{equation}
where%
\begin{equation}
\label{eq:3.20}
\Delta _{v}=\left[ \mu b_{22}-\alpha _{22}+\alpha _{12}\right] ^{2}-4\alpha
_{12}\alpha _{22}\left( K^{2}-1\right) . 
\end{equation}
By choosing $\mu $ such that
\begin{equation}
\label{eq:3.21}
\left[ \mu b_{22}-\alpha _{22}+\alpha _{12}\right] ^{2}-4\alpha _{12}\alpha
_{22}\left( K^{2}-1\right) <0,  
\end{equation}
the quadratic form $\Psi _{v}$ is positive, under the condition
\begin{equation}
\label{eq:3.22}
0<\mu <\frac{\left( \alpha _{22}-\alpha _{12}\right) +2\sqrt{\alpha
_{12}\alpha _{22}\left( K^{2}-1\right) }}{b_{22}}.  
\end{equation}
Finally, we have
\begin{equation}
\label{eq:3.23}
\Psi _{d}\geq -\left( \frac{\left\vert \nabla u\right\vert ^{2}}{8C_{d}}+%
\frac{\left\vert \nabla v\right\vert ^{2}}{8A_{d}}\right) \Delta
_{d}=-\left( \frac{\left\vert \nabla u\right\vert ^{2}}{8d_{2}\mu }+\frac{%
\left\vert \nabla v\right\vert ^{2}}{8\lambda d_{1}}\right) \Delta _{d}, 
\end{equation}
where
\begin{equation}
\label{eq:3.24}
\Delta _{d}=\left( d_{1}+d_{2}\right) ^{2}-4K^{2}d_{1}d_{2}. 
\end{equation}%
Then under Condition  \eqref{eq:1.7}, we can find a neighborhood of $K=1$ such that
for all $\lambda $ and $\mu $ satisfying  \eqref{eq:3.18} and \eqref{eq:3.22}  with $\lambda \mu
=K^{2}$, we have the first inequality in  \eqref{eq:3.3}.
\end{proof}

\begin{remark}
\bigskip We observe that, among the proof of Lemma 3.2, we can prove
\begin{equation}
\label{eq:3.25}
P \nabla H_{u}+Q\nabla H_{v}\geq \lambda _{1}\left( 1+u+v\right) \left(
\left\vert \nabla u\right\vert ^{2}+\left\vert \nabla v\right\vert
^{2}\right) , 
\end{equation}
for all \ $(u,v)\in \Gamma \cap \left\{ (u,v):\ H(u,v)>C_{0}\right\} .$
\end{remark}

\begin{lemma}
\label{lem:l34}
Under Condition \eqref{eq:1.7}, the second inequality in  \eqref{eq:3.3} is satisfied.
\end{lemma}

\begin{proof}
\bigskip For the second Condition  \eqref{eq:3.3} which amount to the positivity of
the following quadratic form in $\nabla u$ and $\nabla v:$%
\begin{equation}
\label{eq:3.26}
\left. 
\begin{array}{rcl}
\Phi &=&\left( H_{u} \cdot  P+H_{v} \cdot Q\right) \nabla H-\lambda _{1}\left( 1+u+v\right)
\left\vert \nabla H\right\vert ^{2} \\ 
&=&A_{1}\left\vert \nabla u\right\vert ^{2}+B_{1}\nabla u \cdot \nabla
v+C_{1}\left\vert \nabla v\right\vert ^{2},%
\end{array}%
\right.   
\end{equation}%
where%
\begin{equation}
\label{eq:3.27}
\left. 
\begin{array}{rcl}
A_{1}&=&\left( H_{u} \cdot P^{u}+H_{v}\cdot Q^{u}\right) H_{u}-\lambda _{1}\left(
1+u+v\right) H_{u}^{2},\  \\ 
B_{1}&=&\left( H_{u}\cdot P^{u}+H_{v}\cdot Q^{u}\right) H_{v}+\left(
H_{u} \cdot P^{v}+H_{v} \cdot Q^{v}\right) H_{u}-2\lambda _{1}\left( 1+u+v\right)
H_{u}H_{v}, \\ 
C_{1}&=&\left( H_{u} \cdot P^{v}+H_{v}\cdot Q^{v}\right) H_{v}-\lambda _{1}\left(
1+u+v\right) H_{v}^{2}.%
\end{array}
\right.  
\end{equation}
and $\lambda _{1}$ any positive constant such that 
\begin{equation}
\label{eq:3.28}
\lambda _{1}<\min \left\{ d_{i},\ \alpha _{ij};\ i,\ j=1,\ 2\right\} . 
\end{equation}
However, a simple calculation shows that its discriminant is given by
\begin{equation}
\label{eq:3.29}
\left. 
\begin{array}{rcl}
\Delta _{1}&=&\left[ \left( H_{u} \cdot P^{u}+H_{v} \cdot Q^{u}\right) H_{v}-\left(
H_{u} \cdot P^{v}+H_{v}\cdot Q^{v}\right) H_{u}\right] ^{2} \\ 
&=&\left[ \left( H_{u}\cdot\left( P^{u}-d_{1}\right) +H_{v} \cdot Q^{u}\right)
H_{v}-\left( H_{u} \cdot P^{v}+H_{v} \cdot \left( Q^{v}-d_{2}\right) \right)
H_{u}+\left( d_{1}-d_{2}\right) H_{u}H_{v}\right] ^{2}.%
\end{array}%
\right.  
\end{equation}
Let us begin with the case $d_{1}=d_{2}:$ The discriminant can be written as
follows 
\begin{equation}
\label{3.30}
\Delta _{1}^{0}=\left\{ \alpha u^{3}+\beta \left[ \left( b_{11}\lambda
-b_{22}\right) u-\left( b_{22}\mu -b_{11}\right) v\right] uv+\gamma
v^{3}\right\} ^{2},  
\end{equation}
where
\begin{equation}
\label{eq:3.31}
\left. 
\begin{array}{rcl}
\alpha &=&\lambda \left[ -b_{11}\lambda +\left( \alpha _{11}-\alpha_{21}\right) \right] ,\\ 
\beta &=&\left( \lambda \mu -1\right) , \\ 
\gamma &=&\left[ b_{22}\mu +\left( \alpha _{12}-\alpha _{22}\right) \right]\mu .
\end{array}
\right.   
\end{equation}
Using  \eqref{eq:3.17} and \eqref{eq:3.21} we can find a positive constant $C_{3}$ such that%
\begin{equation*}
\Delta _{1}\leq C_{3}\left( K^{2}-1\right) H^{3}
\end{equation*}
in a neighborhood of $K=1.$

When $d_{1}\neq d_{2}$, then the discriminant becomes%
\begin{equation}
\label{eq:3.32}
\left. 
\begin{array}{rcl}
\Delta _{1}&=&\left[ \alpha u^{3}+\beta \left[ \left( b_{11}\lambda
-b_{22}\right) u-\left( b_{22}\mu -b_{11}\right) v\right] uv+\gamma
v^{3}+\left( d_{1}-d_{2}\right) H_{u}H_{v}\right] ^{2} \\ 
&\leq& 2\left[ \alpha u^{3}+\beta \left[ \left(
b_{11}\lambda +b_{22}\right) u+\left( b_{22}\mu +b_{11}\right) v\right]
uv+\gamma v^{3}\right] ^{2}+2\left[ \left( d_{1}-d_{2}\right) H_{u}H_{v}%
\right] ^{2}.
\end{array}
\right.   
\end{equation}
Since
\begin{equation*}
\underset{u+v\rightarrow \infty }{\lim }\frac{H_{u}H_{v}}{\sqrt{\Delta
_{1}^{0}}}=0,
\end{equation*}
then for all $\epsilon >0$, we can find positive constants $C_{4}$ and $%
C_{0}$ such that%
\begin{equation}
\label{eq:3.33}
\Delta _{1}\leq C_{4}\left( K^{2}-1+\epsilon \right) H^{3},\text{ for all }%
u+\ v>C_{0}
\end{equation}%
uniformly in a bounded neighborhood of $K=1$. Since%
\begin{equation}
\label{eq:3.34}
\left. 
\begin{array}{rcl}
A_{1}&=&\left[ \left( H_{u} \cdot P^{u}+H_{v}\cdot Q^{u}\right) -\lambda _{1}\left(
1+u+v\right) H_{u}\right] H_{u} \\ 
&=&\left[ \left( \lambda u+v\right) \left( d_{1}-\lambda _{1}+\left(
\alpha _{11}-\lambda _{1}\right) u+\left( \alpha _{12}-\lambda _{1}\right)
v\right) +\left( u+\mu v\right) b_{22}v\right] \left( \lambda u+v\right),
\\ 
C_{1}&=&\left( H_{u} \cdot P^{v}+H_{v} v Q^{v}\right) H_{v}-\lambda _{1}\left(
1+u+v\right) H_{v}^{2} \\ 
&=&\left[ \left( \lambda u+v\right) b_{11}u+\left( u+\mu v\right)
\left( d_{2}-\lambda _{1}+\left( \alpha _{21}-\lambda _{1}\right) u+\left(
\alpha _{22}-\lambda _{1}\right) v\right) \right] \left( u+\mu v\right),%
\end{array}
\right.   
\end{equation}
are positive homogeneous polynomials in $u>0$ and $v>0$ \ of third degree and
as a simple calculation shows that
\begin{equation}
\label{eq:3.35}
C_{5}\left( \left\vert \nabla u\right\vert ^{2}+\left\vert \nabla
v\right\vert ^{2}\right) H\leq \left\vert \nabla H\right\vert ^{2}\leq
C_{6}\left( \left\vert \nabla u\right\vert ^{2}+\left\vert \nabla
v\right\vert ^{2}\right) H,  
\end{equation}%
for some positive constants $C_{5}$ and $C_{6}$, then using inequality
 \eqref{eq:3.14}, we can find another constant $C_{7}$ and a constant $C_{0}$
(independent of $K$ bounded) such that
\begin{equation}
\label{eq:3.36}
\left( H_{u} \cdot P+H_{v} \cdot Q\right) \nabla H\geq \lambda _{1}\left( 1+u+v\right)
\left\vert \nabla H\right\vert ^{2}-C_{7}\left( K^{2}-1+\epsilon \right)
\left( u+v\right) \left\vert \nabla H\right\vert ^{2},
\end{equation}%
for all $u+v>C_{0}$. Taking $\epsilon <\frac{\lambda _{1}}{2C_{7}}$ and a
neighborhood of $K=1$ in which $K^{2}-1<\frac{\lambda _{1}}{2C_{7}},$ we get
the second Condition  \eqref{eq:3.3}.
\end{proof}
\begin{remark}
 We observe that, in the proof of Lemma \ref{lem:l34}, we can prove%
\begin{equation}
\label{eq:3.37}
\left( H_{u} \cdot P+H_{v}  \cdot Q\right) \nabla H\geq \lambda _{1}\left( 1+u+v\right)
\left\vert \nabla H\right\vert ^{2},  
\end{equation}
for all \ $(u,v)\in \Gamma \cap \left\{ (u,v):\ H(u,v)>C_{0}\right\} .$
\end{remark}

Condition  \eqref{eq:3.4}  is assumed by the following

\begin{lemma}
\label{lem:l36}
Condition  \eqref{eq:3.4}  is satisfied.
\end{lemma}

\begin{proof}
We have%
\begin{equation}
\label{eq:3.38}
H_{u}f+H_{v}g=a_{1}u\left( \lambda u+v\right) +a_{2}v\left( u+\mu v\right)
-\Phi \left( u,v\right),  
\end{equation}
with
\begin{equation}
\label{eq:3.381}
\Phi \left( u,v\right) =\lambda b_{1}u^{3}+\left( -\lambda
c_{1}+b_{1}-b_{2}\right) u^{2}v+\left( -c_{1}+c_{2}-\mu b_{2}\right)
uv^{2}+\mu c_{2}v^{3}.  
\end{equation}
Since
\begin{equation*}
\underset{u+v\rightarrow \infty }{\lim }\frac{a_{1}u\left( \lambda
u+v\right) +a_{2}v\left( u+\mu v\right) }{\Phi \left( u,v\right) }=0,
\end{equation*}
then we have  \eqref{eq:3.4} .
\end{proof}

\begin{proof}
\textbf{of Theorem \ref{thm:t21} }
Using the above lemmata, we deduce easily the
inequalities  \eqref{eq:3.3} and \eqref{eq:3.4}  of Theorem \ref{thm:t31} and then the boundedness of the
positive solutions to problem  \eqref{eq:1.1}-\eqref{eq:1.4} under the Condition  \eqref{eq:1.7}.
\end{proof}

\begin{proof}
\textbf{of Theorem \ref{thm:t22}} The proof is an immediate consequence of Theorem
\ref{thm:t11}, Theorem \ref{thm:t21}, and the preliminary observations.
\end{proof}

\section{Numerical Simulations and Discussion}

In this section we would like to support our results in Section 2 by simulating \eqref{eq:1.1}-\eqref{eq:1.4}. Our primary goal is to choose parameters that satisfy Condition \eqref{eq:2.1}, whilst violating \eqref{eq:1.6}, and then simulate our system for a host of different initial conditions. If one observes Case 1 in Table \ref{tab:nonlin}
\begin{equation*}
\alpha _{11}-\alpha _{21} = 0.1 - 0.06 < 0.06 = b_{22},
\end{equation*}
we see Condition (iii) of \eqref{eq:1.6}, is violated, but our Condition (iv)  of \eqref{eq:2.1} holds.

If one observes Case 2 in Table \ref{tab:nonlin}
\begin{equation*}
\alpha _{11}-\alpha _{21} = 1.2 - 0.3 < 1 = b_{22},
\end{equation*}
again we see Condition (iii) of \eqref{eq:1.6}, is violated, but our Condition (iv) of \eqref{eq:2.1} holds.

We perform our numerical simulations in MATLAB R2013b, in two space dimensions with domain $\Omega=[0,\pi]\times[0,\pi]$. The spectral Galerkin method was used to approximate $u,v$ defined in \eqref{eq:1.1}-\eqref{eq:1.2}, as $u^{(n)}:=\sum_{j,k=0}^n\mu_{1,j,k}\varphi_{j,k}$, $v^{(n)}:=\sum_{j,k=0}^n\mu_{2,j,k}\varphi_{j,k}$, where
\begin{equation}
\label{eq:5.1}
\varphi_{j,k}(x,y)=\left\{\begin{array}{rl}
\frac{1}{\pi},&\mbox{if}~j,k=0\\
\frac{\sqrt{2}}{\pi}\cos(ky),&\mbox{if}~j=0,k\neq0\\
\frac{\sqrt{2}}{\pi}\cos(jx),&\mbox{if}~j\neq0,k=0\\
\frac{2}{\pi}\cos(jx)\cos(ky),&\mbox{if}~j,k\neq0
\end{array}\right..
\end{equation}

To approximate $\mu_{i,j,k}$ ($i=1,2$), we use the methods described in \cite{YongThesis} to express $\mu_{1,j,k}$ as the following ordinary differential equation that can be solved numerically using the Matlab function {\tt ode113}:
\begin{eqnarray*}
\frac{d}{dt}\mu_{1,\tilde{j},\tilde{k}}&=&(a_1-(\tilde{j}^2+\tilde{k}^2)d_1)\mu_{1,\tilde{j},\tilde{k}}\nonumber\\
&&-(\tilde{j}^2+\tilde{k}^2)\sum_{l,m=0}^n\sum_{\tilde{l},\tilde{m}=0}^n(\alpha_{11}+\alpha_{12}+b_{11})\mu_{1,l,m}\mu_{2,\tilde{l},\tilde{m}}\int_{\Omega}\varphi_{l,m}\varphi_{\tilde{l},\tilde{m}}\varphi_{\tilde{j},\tilde{k}}\nonumber\\
&&+\sum_{l,m=0}^n\sum_{\tilde{l},\tilde{m}=0}^n(\alpha_{11}\mu_{1,l,m}\mu_{1,\tilde{l},\tilde{m}}+\alpha_{12}\mu_{1,l,m}\mu_{2,\tilde{l},\tilde{m}}+b_{11}\mu_{2,l,m}\mu_{1,\tilde{l},\tilde{m}})\int_{\Omega}\varphi_{l,m}\nabla \varphi_{\tilde{l},\tilde{m}}\cdot\nabla\varphi_{\tilde{j},\tilde{k}}\nonumber\\
&&-\sum_{l,m=0}^n\sum_{\tilde{l},\tilde{m}=0}^n(b_1\mu_{1,\tilde{l},\tilde{m}}+c_1\mu_{2,\tilde{l},\tilde{m}})\mu_{1,l,m}\int_{\Omega}\varphi_{\tilde{j},\tilde{k}}\varphi_{l,m}\varphi_{\tilde{l},\tilde{m}}
\end{eqnarray*}
A similar method is used to approximate $\mu_{2,j,k}$.


We now provide the results of numerical simulations on  \eqref{eq:1.1}-\eqref{eq:1.4}. Two parameter cases are selected (see Table \ref{tab:nonlin}). Notice that Case 1 does indeed fit the form of the SKT model \cite{Shigesada}. Nine simulations were run for each parameter case under various initial conditions for $u$ and $v$ by selecting permutations of the densities given in Figure \ref{fig:sub1}-\ref{fig:sub3}. The final distributions for $u$ and $v$ for both cases under \emph{all} initial conditions approached spatial homogeneity as described in \cite{L96}. This shows that for various initial conditions, our result Theorem \ref{thm:t22} is verified, that is, solutions approach steady state, so exist globally, whilst the results of \cite{D. Le} are not applicable here.

\begin{table}[htp]
\centering  
\begin{tabular}{|c|c|c|} 
\hline  
			& Case \ 1	& Case \ 2		\\ \hline
$d_1$		&  0.01	&   0.25		\\ \hline
$d_2$		&    0.1	&   0.5		\\ \hline
$a_1$		&    1		&   0.2		\\ \hline
$b_1$		&    2		&   0.8		\\ \hline
$c_1$		&   0.2	&   0.8		\\ \hline
$a_2$		&  0.3	&   0.3		\\ \hline
$b_2$		&   1		&   0.4		\\ \hline
$c_2$		&   4		&   0.9		\\ \hline
$\alpha_{11}$	&   0.1	&   1.2		\\ \hline
$\alpha_{12}$	&   0.12	&   0.25		\\ \hline
$\alpha_{21}$	&   0.06	&   0.3		\\ \hline
$\alpha_{22}$	&   0.8	&   0.75		\\ \hline
$b_{11}$		&   0.12	&   0.1		\\ \hline
$b_{22}$		&   0.06	&   1			\\ \hline
\end{tabular}
\caption{Parameters used in simulations} 
\label{tab:nonlin}
\end{table}

\begin{figure}
\centering
\begin{subfigure}{0.33\textwidth}
  \centering
  \includegraphics[width=1\linewidth]{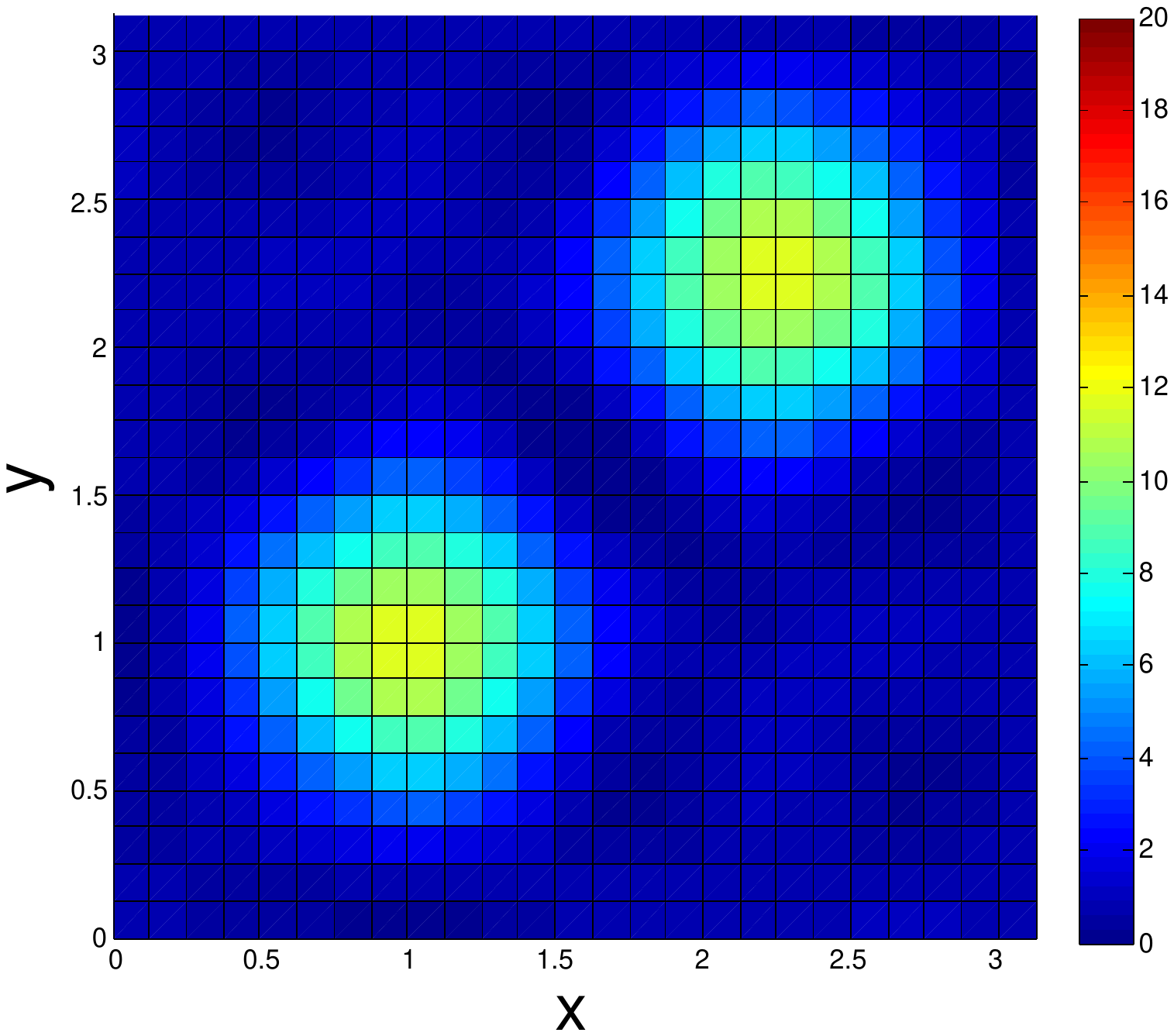}
  \vspace{-2cm}\caption{}
  \label{fig:sub1}
\end{subfigure}%
\begin{subfigure}{0.33\textwidth}
  \centering
  \includegraphics[width=1\linewidth]{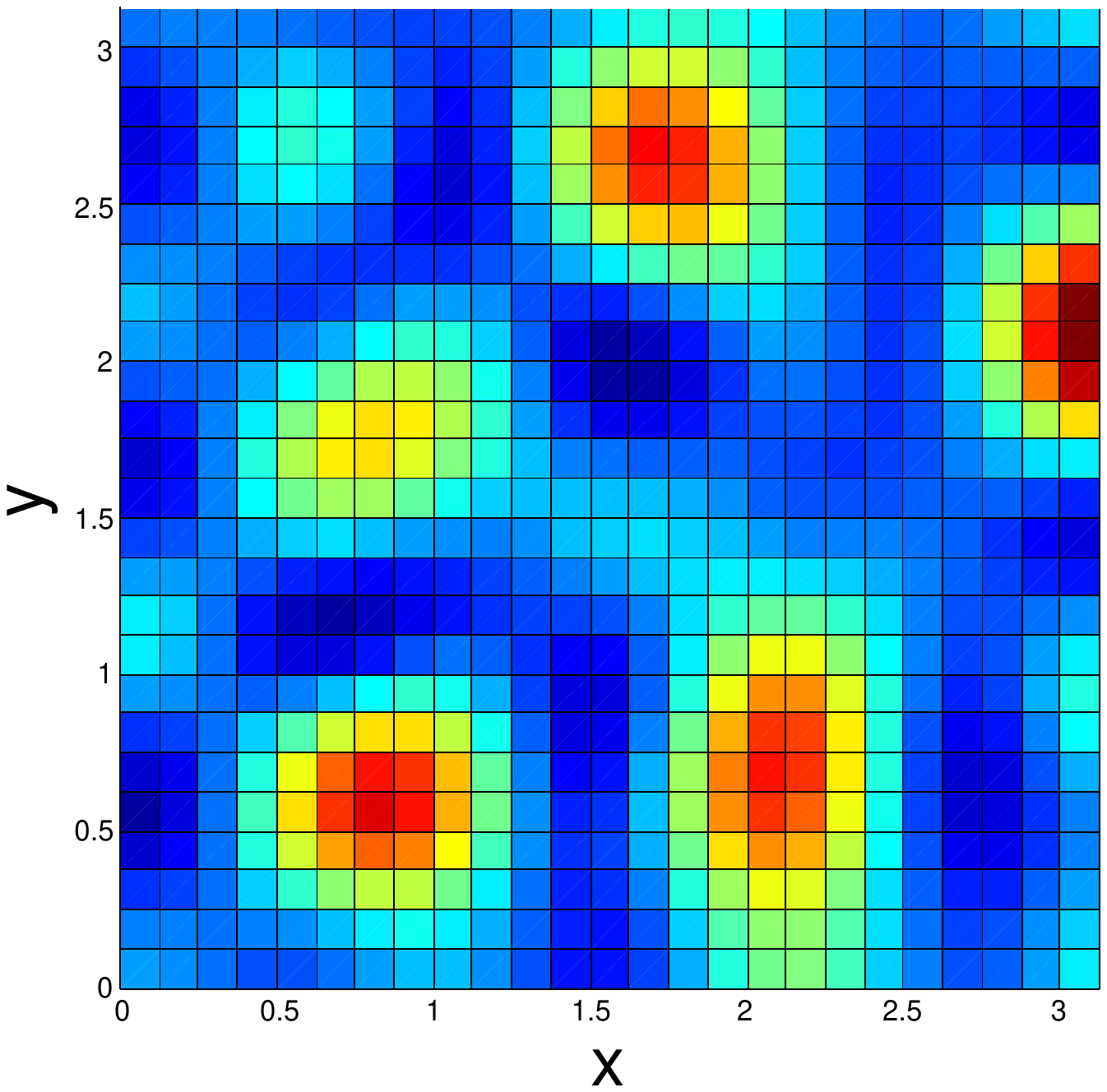}
  \vspace{-2cm}\caption{}
  \label{fig:sub2}
\end{subfigure}
\begin{subfigure}{0.33\textwidth}
  \centering
  \includegraphics[width=1\linewidth]{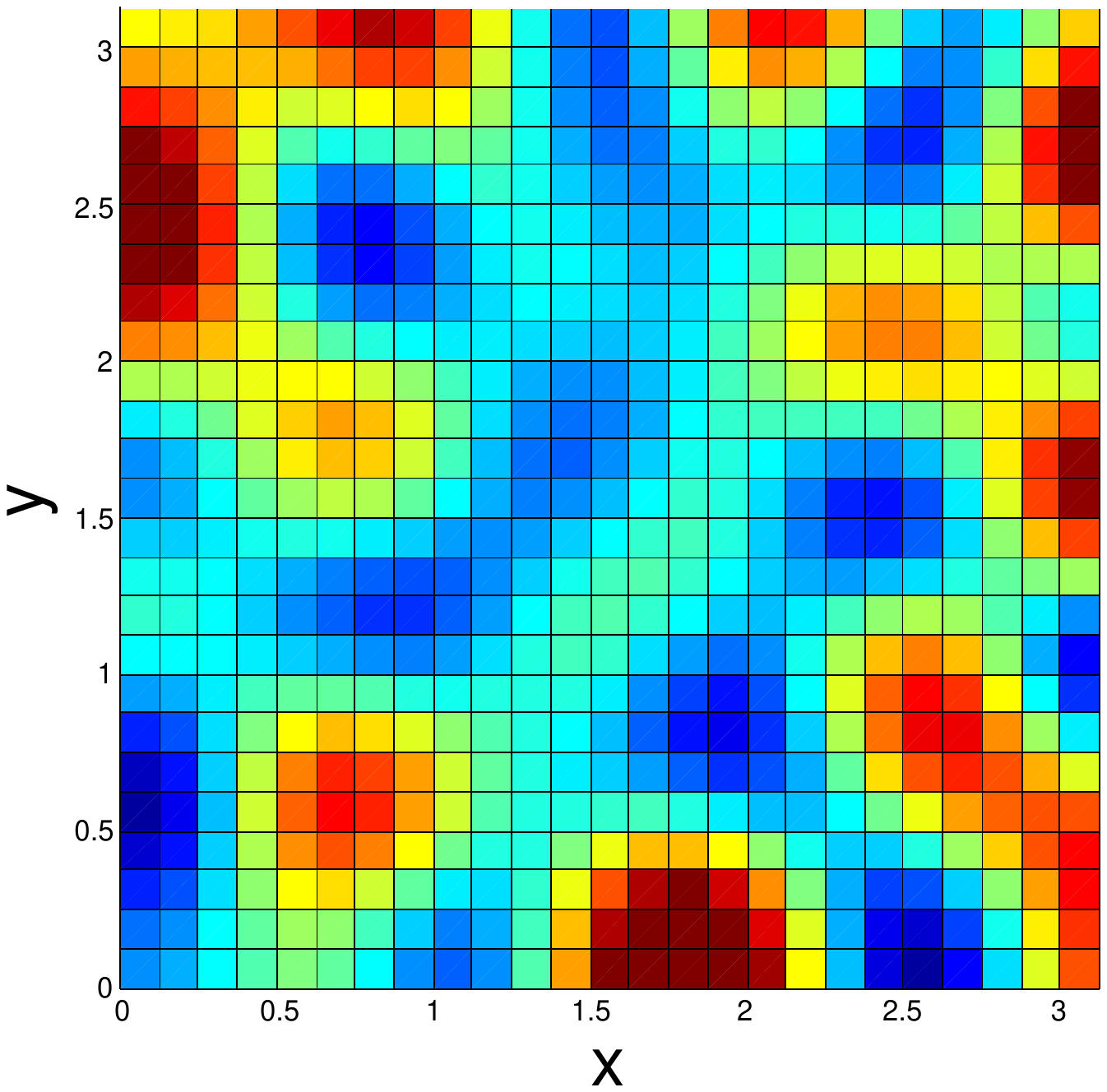}
  \vspace{-2cm}\caption{}
  \label{fig:sub3}
\end{subfigure}

\caption{Figures \ref{fig:sub1}-\ref{fig:sub3} represent possible initial densities of the species $u$ and $v$. We run 9 different tests by selecting initial conditions for $u$ and $v$ from \ref{fig:sub1}, \ref{fig:sub2}, and \ref{fig:sub3}. The color scheme is the same for all plots.}
\label{initial_conditions}
\end{figure}

\section{Acknowledgments}
The second author is partially supported by the NSF grant DMS-0946431.

\end{document}